\documentclass[10pt, a4paper]{article}
\usepackage[width=125mm,height=195mm]{geometry}
\usepackage{amsmath}
\usepackage{amsthm}
\usepackage{amssymb}
\usepackage[arrow, matrix, curve]{xy}
\usepackage{esint}
\usepackage{enumerate}

\theoremstyle{plain}
\newtheorem{prop}{Proposition}
\newtheorem{lemma}{Lemma}
\newtheorem{cor}{Corollary}
\newtheorem{thm}{Theorem}

\theoremstyle{definition}
\newtheorem{defin}{Definition}

\theoremstyle{remark}
\newtheorem*{rem}{Remark}

\newcommand\blfootnote[1]{%
  \begingroup
  \renewcommand\thefootnote{}\footnote{#1}%
  \addtocounter{footnote}{-1}%
  \endgroup
}
  
\newcommand{\address}{{
\bigskip
\footnotesize
 \noindent \textsc{Faculty of Mathematics \\ University of Vienna \newline Oskar-Morgenstern-Platz 1, 1090 Vienna, Austria}\par\nopagebreak
  \noindent \textit{E-mail}: \texttt{christoph.harrach@univie.ac.at}
}}
  
\begin{document}
\title{Poisson transforms for differential forms}
\author{Christoph Harrach}

\newcommand{\mr}{\mathbb{R}}
\newcommand{\xa}{\mathfrak{a}}
\newcommand{\xg}{\mathfrak{g}}
\newcommand{\xk}{\mathfrak{k}}
\newcommand{\xm}{\mathfrak{m}}
\newcommand{\xn}{\mathfrak{n}}
\newcommand{\xp}{\mathfrak{p}}
\newcommand{\xq}{\mathfrak{q}}
\newcommand{\xu}{\mathfrak{u}}
\newcommand{\ce}{\mathcal{E}}
\maketitle

\begin{abstract}
We give a construction of a Poisson transform mapping density valued differential forms on generalized flag manifolds to differential forms on the corresponding Riemannian symmetric spaces, which can be described entirely in terms of finite dimensional representations of reductive Lie groups. Moreover, we will explicitly generate a family of degree-preserving Poisson transforms whose restriction to real valued differential forms has coclosed images. In addition, as a transform on sections of density bundles it can be related to the classical Poisson transform, proving that we produced a natural generalization of the classical theory.
\end{abstract}
\blfootnote{supported by project P27072-N25 of the Austrian Science Fund (FWF)}
\section{Introduction}
In the field of harmonic analysis of Riemannian symmetric spaces $X = G/K$ one of the main tasks in the program proposed by Helgason \cite[II.4]{helgason_GGA} is the explicit determination of the decomposition of the space of smooth functions into joint eigenspaces of the ring of $G$-invariant differential operators. In this context, it was shown that every joint eigenfunction is obtained by the image of a smooth map on the F{\"u}rstenberg boundary $\partial X = K/M$ of $X$ under a $G$-equivariant map, called the Poisson transform. Naturally, it was generalized to a map between sections of natural vector bundles. In this way, it can be used to analyse principal series representations on the F{\"u}rstenberg boundary, which naturally carries the structure of a locally flat parabolic geometry. (cf. \cite{vanderven}, \cite{yang}).

In \cite{gaillard} Gaillard constructed a Poisson transform on the $(n+1)$-dimensional real hyperbolic space $H^{n+1}$ which maps differential forms of degree $k$ on $\partial H^{n+1}$ to differential forms on $H^{n+1}$ of the same degree. This was done in geometric terms on the Poincar{\'e} model following an idea of Thurston in \cite[chapter 11]{thurston}. Afterwards, it was shown that this operator can be expressed as an integral transformation whose kernel is an $\operatorname{SO}(n+1,1)$-invariant differential form on the product space $H^{n+1} \times \partial H^{n+1}$.

In this paper, we will give a more conceptual way to obtain Poisson transforms between differential forms following Gaillard's idea. Explicitly, in section \ref{sec_definition} we will define Poisson transforms between density valued differential forms on the generalized flag manifold $G/P$ and differential forms on the Riemannian symmetric space $G/K$ as an integral transformation, whose kernel is an invariant differential form on a homogeneous space. Such forms are completely determined by invariant vectors in finite dimensional representations of reductive Lie groups and thus can be computed explicitly. Subsequently, in section \ref{sec_construction_poisson_transform} we give an example of a family of such transforms preserving the degree of the differential forms by constructing  Poisson kernels in terms of the $|k|$-grading on the Lie algebra $\xg$. In addition, we will show that the images of these transforms on unweighted forms are coclosed by a simple computation on finite dimensional vector spaces. Finally, we prove that the Poisson transform on sections of density bundles equals the classical Poisson transform, showing that our construction is a natural generalization of the classical theory.

\section{Definition of the Poisson transform}\label{sec_definition}

\subsection{Gaillard's Poisson transform}
Throughout this paper let $\xg$ be a real semisimple Lie algebra and $\theta$ a Cartan involution on $\xg$ with associated Cartan decomposition $\xg = \xk \oplus \xq$. Let $\xa_0$ be a maximal abelian subalgebra of $\xq$ and $\Delta_r$ be the restricted roots of $\xg$. For a restricted root $\alpha$ we denote by $\xg_{\alpha}$ the associated restricted root space. We choose a positive subsystem $\Delta_r^+$ of the restricted roots and let $\xn_0$ denote the direct sum of all $\xg_{\alpha}$ with $\alpha \in \Delta_r^+$. Define the functional $\rho \in \xa_0^*$ via $\rho = \frac{1}{2} \sum_{\alpha \in \Delta^+_r} \dim(\xg_{\alpha})\alpha$. Moreover, let $\xm_0$ be the centralizer of $\xa_0$ in $\xk$. Denote by $\xp_0 = \xm_0 \oplus \xa_0 \oplus \xn_0$ the standard parabolic subalgebra of $\xg$ associated to the choice of positive Weyl chamber and let $\xp$ be a parabolic subalgebra of $\xg$ containing $\xp_0$. Let $G$ be a connected semisimple Lie group with finite center and Lie algebra $\xg$. Let $K \subset G$ be the maximal compact subgroup with Lie algebra $\mathfrak{k}$ and $P$ be a parabolic subgroup of $G$ with Lie algebra $\mathfrak{p}$. Moreover, let $A$ and $N$ be the subgroups of $G$ with Lie algebras $\xa_0$ and $\xn_0$ and denote by $g = k(g) \exp(H(g))n(g)$ the Iwasawa decomposition of $g \in G$, where $k(g) \in K$, $H(g) \in \xa_0$ and $n(g) \in N$.

By the Iwasawa decomposition, the subgroup $K$ acts transitively on $G/P$, which induces an isomorphism $K/M \cong G/P$ with $M := K \cap P$. Therefore, it follows that the $G$-action on the product manifold $G/K \times G/P$ is transitive, so we obtain an isomorphism
\begin{align*}
 \Psi &\colon G/M \to G/K \times G/P, & gM &\mapsto (gK, gP)
\end{align*}
of $G$-manifolds. The derivative of $\Psi$ induces a decomposition of $\xg/\xm$ into the direct sum of the $M$-representations $(\xg/\xm)^{1,0}$ and $(\xg/\xm)^{0,1}$ corresponding to $\xg/\xk$ respectively $\xg/\xp$. Thus, we obtain an induced Whitney sum decomposition of the tangent bundle $T(G/M)$ into the two $G$-invariant subbundles $T^{1,0} := G \times_M (\xg/\xm)^{1,0}$ and $T^{0,1} := G \times_M (\xg/\xm)^{0,1}$. We say that a differential form $\omega$ on $G/M$ has degree $(p,q)$ if it vanishes upon insertion of $(p+1)$ sections of $T^{1,0}$ and $(q+1)$ sections of $T^{0,1}$. In this case, we call $p$ the \emph{$K$-degree} and $q$ the \emph{$P$-degree} of $\omega$.

Let $\ce$ be an oriented homogeneous line bundle over $G/P$ and assume that $G$ acts by orientation preserving maps. Then there is a character $\chi \colon P \to \mr_+$ such that $p \cdot t = \chi(p)t$ for all $t \in \mr$. In this case, we denote the bundle by $\ce[\chi] = G \times_P \mr[\chi]$ and call it the \emph{bundle of $\chi$-densities}. Any section $\sigma \in \Gamma(\ce[\chi])$ corresponds to a smooth, $P$-equivariant map $f \colon G \to \mr[\chi]$. By compactness, the subgroup $M$ has to act trivially on $\mr[\chi]$, so we can view $f$ as an $M$-invariant map. We define the pullback $\pi_P^*\sigma \in C^{\infty}(G/M)$ of the $\chi$-density $\sigma$ as the smooth map $\pi_P^*\sigma(gM) := f(g)$ for all $g \in G$.

For any $\ce[\chi]$-valued differential form $\alpha$ on $G/P$ of degree $k$ we can form its pullback along the canonical projection $\pi_P \colon G/M \to G/P$, which is a differential form on $G/M$ of degree $(0,k)$. Thus, for any $\phi_{k, \ell} \in \Omega^{\ell,n-k}(G/M)$ with $n = \dim(G/P)$ the integral of the wedge product $\phi_{k,\ell} \wedge \pi_P^*\alpha$ over the compact fiber $G/P$ is a differential form $\Phi(\alpha)$ on $G/K$ of degree $\ell$. By definition, the induced map $\Phi$ will be $G$-equivariant if and only if $\phi_{k,\ell}$ is $G$-invariant.

\begin{defin} Let $\phi_{k,\ell}$ be a $G$-invariant differential form of degree $(\ell, n-k)$ on $G/M$, where $n = \dim(G/P)$. We call the induced $G$-equivariant operator
\begin{align*}
 \Phi \colon \Omega^k(G/P, \ce[\chi]) &\to \Omega^{\ell}(G/K), & \alpha &\mapsto \fint_{G/P} \phi_{k,\ell} \wedge \pi_P^*\alpha
\end{align*}
a \emph{Poisson transform} and the corresponding differential form $\phi_{k,\ell}$ its \emph{Poisson kernel}.
 \end{defin}

Therefore, Poisson transforms mapping density valued $k$-forms on $G/P$ to $\ell$-forms on $G/K$ correspond to $G$-invariant differential forms on $G/M$ of bidegree $(\ell, n-k)$. As invariant sections of a homogeneous vector bundle, these are in bijective correspondence with $M$-invariant elements in the finite dimensional $M$-representation $\Lambda^{\ell, n-k}(\xg/\xm)^*$, where $\xm$ denotes the Lie algebra of $M$, see \cite[Theroem 1.4.4]{cap_slovak}. 

\begin{prop}
 Poisson kernels on $G/M$ are in bijective correspondence with the ring of $M$-invariant elements in $\Lambda^* (\xg/\xm)^*$.
\end{prop}

As a result, we are able to construct Poisson transforms in a purely algebraic way via computations in finite dimensional $M$-representations. On the one hand, as the group $M$ is reductive, there is a wide variety of tools available to analyse $M$-modules quite efficiently. On the other hand, the ring of possible Poisson kernels depends strongly on the choices of $\xg$ and $\xp$ and thus cannot be uniformly determined in full generality. However, in the next section we will explicitly construct a Poisson transform $\Phi_k \colon \Omega^k(G/P, \ce[\chi]) \to \Omega^k(G/K)$ and show that on unweighted forms its image consists of coclosed forms. Furthermore, in the case $k = 0$ it reduces to the classical Poisson transform as in \cite[II.3.4]{helgason_GASS}.

\section{Construction of Poisson transforms}\label{sec_construction_poisson_transform}
In this section we will analyse the structure of the tangent bundle of $G/M$ and subsequently construct a Poisson transform mapping between differential forms, which exists for each choice of $G$ and $P$.

\subsection{Decomposition of the tangent bundle $T(G/M)$}
Recall \cite[chapter 3.2]{cap_slovak} that the choice of a parabolic subalgebra $\xp$ induces a $|k|$-grading 
\begin{align*}
 \xg = \xg_{-k} \oplus \dotso \oplus \xg_k
\end{align*}
on $\xg$ so that $\xp$ is the direct sum of the nonnegative grading components. We define $\xg_-$ and $\xp_+$ to be the direct sum of all negative and positive grading components, respectively. These are nilpotent subalgebras of $\xg$ due to the compatibility condition of the grading with the Lie bracket. Each grading component $\xg_i$ is by definition a $\xg_0$-module via the adjoint action. We denote by $E$ the grading element of $\xg$, i.e. the unique element in the center of $\xg_0$ which satisfies $[E, X] = iX$ for all $X \in \xg_i$. Moreover, the Killing form $B$ of $\xg$ induces a nondegenerate pairing of $\xg_{i}$ and $\xg_{-i}$, and the Cartan involution maps $\xg_i$ to $\xg_{-i}$ for all $i = 0, \dotsc, k$. In particular, the subalgebra $\xg_0$ of $\xg$ is $\theta$-stable and thus can be decomposed into the direct sum of $\xg_0 \cap \xq$ and  $\xg_0 \cap \xk$. Since $\xp_+$ is the direct sum of positive root spaces its intersection with $\xk$ is trivial, implying that $\xg_0 \cap \xk = \xm$.

\begin{prop}\label{prop_decomposition_tangent_bundle}
The tangent bundle of $T(G/M)$ decomposes into a Whitney sum of $G$-invariant subbundles
\begin{align*}
 T(G/M) \cong T^{-k}(G/M) \oplus \dotsb \oplus T^k(G/M).
\end{align*}
For $i \neq 0$, the rank of the bundle $T^i(G/M)$ equals the dimension of the $i$-th grading component, whereas the rank of $T^0(G/M)$ equals $d := \dim(G/K) - \dim(G/P)$. Moreover, the bundles $T^{i}(G/M)$ and $T^{-i}(G/M)$ are canonically isomorphic via the map $\Theta_M$ induced by the Cartan involution for all $i = 0, \dotsc, k$.

In this picture, the invariant subbundles $T^{1,0}$ and $T^{0,1}$ are given by
 \begin{align*}
 T^{1,0} &= \bigoplus_{i=0}^k T^i(G/M),  &  T^{0,1} &= \bigoplus_{i=1}^k S_{\Theta_M}(T^{-i}(G/M)),
 \end{align*}
where $S_{\Theta_M}(T^{-i}(G/M)) = \{X + \Theta_M(X) \mid X \in T^{-i}(G/M)\}$ denotes the symmetrization with $\Theta_M$ for all $i = 1, \dotsc, k$.
\end{prop}

\begin{proof}
The image of the $M$-module $\xg_i$ under the $M$-equivariant projection $\xg \to \xg/\xm$ is again an $M$-representation $(\xg/\xm)_i$ for all $i = -k, \dotsc, k$. Therefore, we get a decomposition 
\begin{align*}
 \xg/\xm  = (\xg/\xm)_{-k} \oplus \dotso \oplus (\xg/\xm)_{k}
\end{align*}
into $M$-modules. Since $\xm$ is contained in $\xg_0$ we deduce that every component $(\xg/\xm)_i$ is isomorphic to $\xg_i$ for $i \neq 0$, whereas $(\xg/\xm)_0$ is isomorphic to $\xg_0 \cap \xq$. Moreover, the Lie algebra $\xm$ is $\theta$-stable, so the Cartan involution factorizes to a map $\theta_{\xm} \colon \xg/\xm \to \xg/\xm$ which maps $(\xg/\xm)_{-i}$ to $(\xg/\xm)_{i}$ for all $i$.

By definition of the isomorphism $\Psi$ we see that $(\xg/\xm)^{1,0}$ equals $\xp/\xm$, whereas the subrepresentation $(\xg/\xm)^{0,1}$ coincides with $\xk/\xm$. Since the subalgebra $\xp$ is given by the sum of all nonnegative grading components we immediately obtain that $\xp/\xm = \bigoplus_{i=0}^k (\xg/\xm)_i$. On the other hand, we can determine $\xk$ as the $+1$-eigenspace of the Cartan involution. Since $\theta$ maps $\xg_i$ to $\xg_{-i}$ we obtain that the algebra $\xk$ is given by the sum of $\xm$ together with the $\theta$-symmetrizations $S_\theta(\xg_{-i}) = \{X + \theta(X) \mid X \in \xg_{-i}\}$ for all $i > 0$. Moreover, $\xm$ is $\theta$-stable, so we see that the projection of $S_{\theta}(\xg_{-i})$ to $\xg/\xm$ coincides with the symmetrization $S_{\theta_{\xm}}((\xg/\xm)_{-i})$, which is defined analogously. 

For the global picture, recall that the tangent bundle $T(G/M)$ is the homogeneous vector bundle associated to $\xg/\xm$. Each of the subspaces $(\xg/\xm)_i$ is $M$-invariant, so they give rise to $G$-invariant subbundles $T^i(G/M) \subset T(G/M)$ via the associated bundle construction for $i = -k, \dotsc, k$. Moreover, the Cartan involution $\theta_{\xm}$ induces a $G$-equivariant bundle map $\Theta_M \colon T(G/M) \to T(G/M)$, which restricts to a canoncial isomorphism $T^{i}(G/M) \cong T^{-i}(G/M)$ of homogeneous vector bundles for all $i$. Finally, note that the spaces $S_{\theta_{\xm}}((\xg/\xm)_{-i})$ are again $M$-modules, and their induced bundles coincide with $S_{\Theta_M}(T^{-i}(G/M))$ by naturality of the associated bundle construction.
\end{proof}

\subsection{A Poisson transform between differential forms}
The last part of this paper is dedicated to the construction of a familiy of Poisson transforms $\Phi_k$ for all $k = 0, \dotsc, n$ mapping $k$-forms on $G/P$ to differential forms on $G/K$ of the same degree. For that, we will construct an invariant $(1,0)$-form on the $M$-representation $(\xg/\xm)_0$, which will serve as a building block for Poisson kernels of higher degree. In order to obtain induced invariant forms we first relate the exterior derivatives on $G/K$ and $G/P$ as well as the Riemannian codifferential to Poisson transforms.

The bigrading of the space of differential forms on $G/M$ induces a natural splitting of the exterior derivative into the sum of partial derivatives $d = d_K + d_P$, where the first operator raises the $K$-degree and the second operator raises the $P$-degree. We call these operators the \emph{$K$-derivative} and the \emph{$P$-derivative} respectively. Since $d$ is a differential, so are both partial derivatives and in addition the relation $d_Kd_P = -d_Pd_K$ holds.

Next, we consider the Hodge star operator $\ast$ on $p$-forms on $G/K$, which is induced by a $K$-equivariant map $\Lambda^p(\xg/\xk)^* \to \Lambda^{n+d-p}(\xg/\xk)^*$. After applying the $M$-invariant isomorphism $\xg/\xk \cong (\xg/\xm)^{1,0}$ on both sides we can form the tensor product with the $q$-th exterior power of $(\xg/\xm)^{0,1*}$, obtaining an $M$-equivariant map $\ast_K \colon \Lambda^{p,q}(\xg/\xm)^* \to \Lambda^{n+d-p, q}(\xg/\xm)^*$. We call the induced map
\begin{align*}
 \ast_K \colon \Omega^{p,q}(G/M) \to \Omega^{n+d-p,q}(G/M)
\end{align*}
the \emph{$K$-Hodge star}. Using this, we define the \emph{$K$-codifferential} on forms on $G/M$ by the formula $\delta_K = (-1)^{(n+d)(p-1)+1} \ast_K d_K \ast_K$, which coincides with the pullback of the Riemannian codifferential $\delta$ on $G/K$ along the canonical projection.

By naturality, all the above maps are $G$-equivariant and therefore have $M$-equivariant counterparts on the level of the corresponding $M$-representations, which we will denote by the same symbols. 

\begin{prop}\label{prop_Poisson_operators}
Let $\Phi \colon \Omega^k(G/P) \to \Omega^{\ell}(G/K)$ be a Poisson transform with corresponding Poisson kernel $\phi$ of degree $(\ell, n-k)$. Then $d \circ \Phi$, $\delta \circ \Phi$ and $\Phi \circ d$ are also Poisson transforms with kernels $d_K\phi$, $\delta_K\phi$ and $(-1)^{n-k+\ell+1} d_p\phi$ respectively.
\end{prop}
\begin{proof}
  Let $\alpha$ be a $k$-form on $G/P$. Since the exterior derivative commutes with the fiber integral (\cite[ch. VII, Prop. X]{greub_halperin_vanstone}) we obtain that
  \begin{align*}
   d\Phi(\alpha) = \fint_{G/P}\!d(\phi \wedge \pi_P^* \alpha) = \fint_{G/P}\!(d\phi) \wedge \pi_P^*\alpha + (-1)^{n-k+\ell} \phi \wedge (d\pi_P^*\alpha).
  \end{align*}
Using naturality of the exterior derivative in the second summand we see that the integrand $\phi \wedge \pi_P^*(d\alpha)$ is a form on $G/M$ of bidegree $(\ell, n+1)$ and hence vanishes. For the first summand, we split $d$ into partial derivatives, where similarly the $d_P$-part is trivial due to the $P$-degree of the integrand. All in all we get
\begin{align*}
 d\Phi(\alpha) = \fint_{G/P}\! (d_K\phi) \wedge \pi_P^*\alpha.
\end{align*}
The definition of the $K$-Hodge star implies that $\ast \circ \fint = \fint \circ \ast_K$, so we immediately see that $\delta \circ \Phi$ is a Poisson transform with kernel $\delta_K\phi$.

Finally, let $\alpha$ be a $(k-1)$-form on $G/P$. Naturality and the antiderivation property of the exterior derivative imply that
\begin{align*}
 \Phi(d\alpha) = \fint_{G/P}\! \phi\wedge d\pi_P^*\alpha = \fint_{G/P}(-1)^{|\phi|} \left( d(\phi \wedge \pi_P^*\alpha) - (d\phi) \wedge \pi_P^*\alpha\right).
\end{align*}
Commuting the exterior derivative and the fiber integral in the first summand, the resulting integrand is a form on $G/M$ of degree $(\ell, n-1)$, whose fiber integral is trivial. For the other summand, we split $d$ into partial derivatives, and the summand containing the $K$-derivative is again trivial due to dimensional reasons.
\end{proof}

In order to keep track of the bidegree of forms on $G/M$, we will introduce representatives of the $M$-subrepresentations $(\xg/\xm)^{1,0}$ and $(\xg/\xm)^{0,1}$, compare with Proposition \ref{prop_decomposition_tangent_bundle}. For all $X \in \xp_+$ we denote by $F_X$ the vector $X + \xm \in (\xg/\xm)^{1,0}$. Similarly, for all $X \in \xg_-$ we define $G_X \in (\xg/\xm)^{0,1}$ via $G_X := S_{\theta_{\xm}}(X)+\xm$. 

\begin{prop}\label{prop_invariant_form_general}
 The grading element on $\xg$ induces a nontrivial $M$-invariant element $E^* \in (\xg/\xm)^*$ of degree $(1,0)$. Moreover, its $K$-derivative is trivial, whereas its $P$-derivative is given by
 \begin{align*}
  d_PE^*(F_X, G_Y) = iB(E,E)^{-1}B(X,Y)
 \end{align*}
for all $i > 0$, $X \in \xg_i$ and $Y \in \xg_{-i}$. In particular, its restriction to the subspace $(\xg/\xm)_{-i} \times (\xg/\xm)_i$ is nondegenerate and $M$-invariant.
\end{prop}

\begin{proof}
 The grading element $E$ of $\xg$ is contained in the center of $\xg_0$, so in particular it commutes with the subalgebra $\xm$. Moreover, for all $X \in \xg_i$ we have
\begin{align*}
 [\theta(E), X] = \theta([E, \theta(X)]) = -iX,
\end{align*}
which implies that $\theta(E) = -E$ and therefore $E \in \xq$. Since $\xg_0$ is $\theta$-stable, this shows that $E$ is contained in the subalgebra $\xg_0 \cap \xq$, which is complementary to $\xm$. Therefore, we can project the grading element to a distinguished element in $(\xg/\xm)_0$, which we will denote by the same symbol. Since the grading element is in the centralizer of $\xm$, its dual $E^* \in (\xg/\xm)^*$ is $M$-invariant and by construction of degree $(1,0)$. 

By definition, the image of $E^*$ under the $M$-equivariant map corresponding to the exterior derivative on $G/M$ can be computed via
\begin{align*}
 dE^*(X + \xm, Y + \xm) = -E^*([X,Y] + \xm)
\end{align*}
for all $X$, $Y \in \xg$, which is well defined due to $M$-invariance of $E^*$. Moreover, the compatibility of the Lie bracket on $\xg$ with the $|k|$-grading shows that the derivative is trivial unless $X \in \xg_{-i}$ and $Y \in \xg_i$ for all $i = -k, \dotsc, k$. Therefore, we have to determine explicitly the projections of the Lie brackets to the subspace generated by the grading element.

For two vectors $X$ and $Y$ in $\xg_0 \cap \xq$ the Lie bracket $[X,Y]$ is in $\xg_0 \cap \xk = \xm$, which implies that $dE^*$ is trivial on $\xg_0$. For the other components, recall that the decomposition of $\xg_0$ into the direct sum $(\xg_0 \cap \xq) \oplus \xm$ is orthogonal with respect to $B$ and the restriction of the Killing form to $\xg_0 \cap \xq$ is positive definite. For $X \in \xg_{-i}$ we choose $Y \in \xg_i$ with $B(X, Y) \neq 0$. By the invariance property of the Killing form we obtain that
\begin{align*}
 B(E, [X, Y]) = B([E, X], Y) = -iB(X,Y) \neq 0.
\end{align*}
We can extend $E = E_1$ to an orthogonal basis $\{E_1, E_2, \dotsc, E_\ell\}$ of $\xg_0 \cap \xq$. Then the dual basis with respect to the Killing form is given by $\{\lambda_1E_1, \dotsc, \lambda_\ell E_\ell\}$, where $\lambda_i = B(E_i,E_i)^{-1} \neq 0$, and every vector $Z \in \xg_0 \cap \xq$ can be expressed as the sum $Z = \sum_{i=1}^\ell \lambda_i B(E_i, Z)E_i$. Therefore, for $X \in \xg_{-i}$ and $Y \in \xg_i$ we obtain from above that
\begin{align*}
 dE^*(X + \xm,Y + \xm) = -E^*([X,Y] + \xm) = i\lambda_1B(X,Y) \neq 0.
\end{align*}
This shows that the restriction of $dE^*$ to $(\xg/\xm)_{-i} \times (\xg/\xm)_{i}$ is nondegenerate for all $1 \le i \le k$. Moreover, the derivative is trivial on $(\xg/\xm)^{1,0}$, hence the $K$-derivative of $E^*$ is trivial.
\end{proof}

The $M$-invariant $(1,0)$-form constructed in the last Proposition can be used to produce Poisson kernels of higher degree by applying differential operators and forming wedge products. In order to define a Poisson transform between differential forms preserving the degrees, we need to make one more observation.

\begin{lemma}\label{lem_invariant_form_exterior}
 Assume that $M$ acts trivially on the $1$-dimensional representation $\Lambda^d (\xg/\xm)_0^*$. Then the $K$-derivative of every element $\tau$ in this representation is trivial.
\end{lemma}
\begin{proof}
 In order to compute the $K$-derivative of $\tau$ we have to insert vectors of the form $[X,Y] + \xm$ with $X$, $Y \in \xp$ into $\tau$. However, due to the properties of the $|k|$-grading we see that this vector is in $(\xg/\xm)_0$ if and only if $X$ and $Y$ are in $\xg_0 \cap \xq$, in which case $[X,Y] \in \xm$.
\end{proof}
In particular, note that the assumption of the Lemma is satisfied if $M$ is connected or if $P$ is the minimal parabolic, in which case $M$ acts trivially on $(\xg/\xm)_0 \cong \xa_0$.

\begin{thm}\label{thm_poisson_transform}
Let $G$ be a semisimple Lie group with finite center, $K \subset G$ a maximal compact subgroup, $P$ a parabolic subgroup of $G$ and $M := K \cap P$. Let $\xg$ and $\xm$ be the Lie algebras of $G$ respectively $M$ and let $\xg_0$ be the $0$-th grading component of the $|k|$-grading of $\xg$ induced by $P$. Let $(\xg/\xm)_0$ be the projection of $\xg_0$ to $\xg/\xm$ and assume that $M$ acts trivially on the representation $\Lambda^{d}(\xg/\xm)_0^*$, where $d$ is the dimension of $(\xg/\xm)_0$.

Then for all $k = 0, \dotsc, \dim(G/P)$ there is a Poisson transform 
  \begin{align*}
  \Phi_k \colon \Omega^k(G/P, \ce[\chi]) \to \Omega^{k}(G/K).
  \end{align*}
  Moreover, on unweighted forms its image is coclosed.
\end{thm}

\begin{proof}
 From Proposition \ref{prop_invariant_form_general} and Lemma \ref{lem_invariant_form_exterior} we obtain two $M$-invariant forms $E^*$ and $\tau$ on $\xg/\xm$ and we denote the induced $G$-invariant differential forms on $G/M$ with the same symbol. Then we can form the Poisson kernel
\begin{align*}
 \phi_k := \ast_K(\tau \wedge (d_PE^*)^{n-k}),
\end{align*}
which by definition is an invariant form of degree $(k,n-k)$ and thus induces a Poisson transform 
\begin{align*}
\Phi_k \colon \Omega^k(G/P, \ce[\chi]) \to \Omega^{k}(G/K). 
\end{align*}
Moreover, we have seen that the $K$-derivative of $E^*$ and $\tau$ is trivial and thus also $d_K(d_PE^*) = -d_Pd_KE^* = 0$. For the Poisson kernel $\phi_k$ this implies that $\delta_K\phi_k = 0$ for all $k$. 
\end{proof}

In the special case of $G = \operatorname{SO}(n+1,1)$ one can show that $\Phi_k$ is the only Poisson transform (up to a multiple) preserving the degree of the differential forms and therefore has to coincide with the transform constructed by Gaillard in \cite{gaillard}. In this case we can generate the ring of Poisson kernels from $E^*$ by applying differential operators and wedge products. In particular, this shows that we cannot expect any other Poisson kernels in general.

On the other hand, Proposition \ref{prop_invariant_form_general} suggests that the ring of Poisson kernels will have a much more complicated structure for other choices of Lie algebras. Indeed, first we could restrict the derivative $d_PE^*$ to the subspaces $(\xg/\xm)_i \times (\xg/\xm)_{-i}$ and extend them trivially, obtaining a family of $M$-invariant $(1,1)$-forms on $\xg/\xm$, whose number depends on the length of the grading on $\xg$. Moreover, every other central element in $\xg_0\cap \xq$ induces an invariant $1$-form on $\xg/\xm$ in the same way as the grading element. Finally, there might be additional $G$-invariant geometric structures on $G/K$ and $G/P$, which are reflected by the existence of additional $M$-invariant forms on some grading components.

As a final result we show that in the case of the minimal parabolic subgroup we obtain the classical Poisson transform as defined in \cite[ch.II. §3.4]{helgason_GASS}. Let $P$ be the parabolic subgroup of $G$ corresponding to $\xp_0$. In this case, the $|k|$-grading of $\xg$ satisfies $\xp_+ = \xn_0$, $\xg_0 \cap \xq = \xa_0$ and $\xg_0 \cap \xk = \xm_0$. In particular, the group $M$ acts trivially on $(\xg/\xm)_0$, so Lemma \ref{lem_invariant_form_exterior} is applicable. The Poisson kernel $\phi_0$ is an $G$-invariant differential form on $G/M$ of degree $(0,n)$, so we can normalize it so that the fibers of $G/M \to G/K$ have unit volume.

Let $\chi \colon P \to \mr_+$ be any character of $P$. Since $M$ is compact and $N$ is unipotent, they have trivial image under $\chi$. Therefore, the character is determined by its image on the abelian subgroup $A$, so in turn by a linear functional $\lambda \in \xa_0^*$. Explicitly, this correspondence is given by $\chi(p) = e^{\lambda(H(p))}$, where $H(p)$ is the $\xa_0$-part in the Iwasawa decomposition of $p \in P$. In this case, we denote the corresponding density bundle by $\ce[\lambda]$ and call it the \emph{bundle of $\lambda$-densities}.
\begin{cor}
  The Poisson transform $\Phi_0 \colon \Gamma(\ce[\rho - \lambda]) \to C^{\infty}(G/K)$ defined in Theorem \ref{thm_poisson_transform} is the classical Poisson transform.
\end{cor}
\begin{proof}
By definition, the fiber of $\pi_K$ over $gK$ is given by the set $\{gkM \mid k \in K\}$. Thus, applying the transformation formula for the action $\ell_g$ of $G$ on $G/M$ we obtain for all $\sigma \in \Gamma(\ce[\rho - \lambda])$ that
\begin{align*}
 \Phi_0(\sigma)(gK)  &= \int_{\pi_K^{-1}(gK)} \pi_P^*\sigma \ \phi_0 = \int_{K/M} \ell_g^*\pi_P^*\sigma \ \phi_0 = \int_{K/M} (\pi_P^*\sigma)(gkM)\ dkM,
\end{align*}
where $dkM$ is the normalized, $K$-left invariant measure on $K/M$, see \cite[Theorem 1.9]{helgason_GGA}. Using the definition of $\pi_P^*\sigma$ we can replace the above integrand by $e^{(\lambda-\rho)(H(gk))}\pi_P^*\sigma(k(gk)M)$. But the map $T_g \colon k \mapsto k(gk)$ is a diffeomorphism of $K$ with inverse $T_g^{-1} = T_{g^{-1}}$ (see \cite[Lemma 5.19]{helgason_GGA}), and the change of measure is given by $T_g^*dk = e^{-2\rho(H(gk))} dk$. Inserting this into the above formula and using that $H(gk(g^{-1}k)) = -H(g^{-1}k)$ we obtain that
\begin{align*}
 \Phi_0(\sigma)(gK) = \int_{K/M} e^{-(\lambda+\rho)(H(g^{-1}k))} (\pi_P^*\sigma)(kM) dkM,
\end{align*}
which is the classical Poisson transform.
\end{proof}
 \begin{rem}
  The classical Poisson transform is usually defined on smooth functions on $K/M$ rather than on sections of density bundles on $G/P$. However, consider a smooth map $\phi$ on $K/M$ and let $\tilde{\phi} \colon K \to \mr$ be the corresponding $M$-equivariant map defined by $\tilde{\phi}(k) := \phi(kM)$. Then the map $f \colon G \to \mr[\lambda]$ given by $f(g) := e^{\lambda(H(g))}\tilde{\phi}(k(g))$ is $P$-equivariant and therefore corresponds to a $\lambda$-density on $G/P$. Thus, for all $\lambda \in \xa_0^*$ we obtain an isomorphism 
\begin{align*}
 \chi_{\lambda}\colon C^{\infty}(K/M) \to \Gamma(\ce[\lambda])
\end{align*}
which satisfies $(\pi_P^* \circ \chi_{\lambda} \circ f)(kM) = f(kM)$ for all $kM \in K/M$. Using this identification we see that $\Phi_0(\chi_{\rho -\lambda}f)$ coincides with the definition in \cite[II.3.4]{helgason_GASS}. 
 \end{rem}

 \bibliographystyle{amsplain}
\providecommand{\MRhref}[2]{%
  \href{http://www.ams.org/mathscinet-getitem?mr=#1}{#2}
}
\providecommand{\href}[2]{#2}

\address
\end{document}